\documentclass[12pt,reqno]{amsart}

\newcommand\version{July 7, 2022}

\usepackage{amsmath,amsfonts,amsthm,amssymb,amsxtra}
\usepackage{bbm} 
\usepackage{esint} 

\newtheorem{theorem}{Theorem}

\newtheorem{lemma}[theorem]{Lemma}
\newtheorem{corollary}[theorem]{Corollary}

\theoremstyle{definition}

\theoremstyle{remark}

\newtheorem{remark}[theorem]{Remark}

\newcommand{\1}{\mathbbm{1}}

\renewcommand{\epsilon}{\varepsilon}

\newcommand{\loc}{{\rm loc}}

\renewcommand{\phi}{\varphi}
\newcommand{\R}{\mathbb{R}}

\newcommand{\Sph}{\mathbb{S}}

\newcommand{\Z}{\mathbb{Z}}

\DeclareMathOperator{\dist}{dist}

\DeclareMathOperator{\sgn}{sgn}

\begin{document}

\title[A characterization of $\dot W^{1,p}(\R^d)$ --- \version]{A characterization of $\dot W^{1,p}(\R^d)$}

\author{Rupert L. Frank}
\address[Rupert L. Frank]{Mathe\-matisches Institut, Ludwig-Maximilans Universit\"at M\"unchen, The\-resienstr.~39, 80333 M\"unchen, Germany, and Munich Center for Quantum Science and Technology, Schel\-ling\-str.~4, 80799 M\"unchen, Germany, and Mathematics 253-37, Caltech, Pasa\-de\-na, CA 91125, USA}
\email{r.frank@lmu.de}

\renewcommand{\thefootnote}{${}$} \footnotetext{\copyright\, 2022 by the author. This paper may be reproduced, in its entirety, for non-commercial purposes.\\
	Partial support through U.S.~National Science Foundation grant DMS-1954995 and through the German Research Foundation grant EXC-2111-390814868 is acknowledged. The author is grateful to Ha\"im Brezis, Mario Milman, Fedor Sukochev, Jean Van Schaftingen, Po-Lam Yung and Dmitriy Zanin, as well as an anonymous referee for many helpful suggestions.}

\begin{abstract}
	For $1<p<\infty$ we give a characterization of the Sobolev space $\dot W^{1,p}(\R^d)$ in terms of the oscillations of a function on balls of varying centers and radii. Our work is motivated both by the study of trace ideal properties of commutators with singular integral operators and by work of Nguyen and by Brezis, Van Schaftingen and Yung on derivative-free characterizations of Sobolev spaces.
\end{abstract}

\dedicatory{Dedicated, in admiration, to V.\ Maz'ya on the occasion of his 85th birthday}

\maketitle

\section{Main result and discussion}

By definition, the space $\dot W^{1,p}(\R^d)$ consists of all $f\in L^1_\loc(\R^d)$ whose distributional gradient satisfies $\nabla f\in L^p(\R^d)$. (This space is denoted by $L^1_p(\R^d)$ in \cite{Ma}.) Our goal here is to discuss a necessary and sufficient criterion for the membership to $\dot W^{1,p}(\R^d)$ in the case $p>1$ that does not involve derivatives. Throughout this paper, $d\geq 1$.

We denote $\R^{d+1}_+:=\R^d\times\R_+$ and, for $(a,r)\in\R^{d+1}_+$, we set $B_r(a):=\{ x\in\R^d:\ |x-a|<r\}$. For a function $f\in L^1_\loc(\R^d)$, let
$$
m_f(a,r) := \fint_{B_r(a)} \left| f(x) - \fint_{B_r(a)} f(y)\,dy \right|dx \,,
$$
where $\fint_{B_r(a)} \ldots = |B_r(a)|^{-1} \int_{B_r(a)}\ldots$. Finally, let $\nu_p$ be the measure on $\R^{d+1}_+$ with
$$
d\nu_p(a,r) = \frac{da\,dr}{r^{p+1}} \,.
$$

Our main result is the following.

\begin{theorem}\label{main}
	Let $1<p<\infty$ and let $f\in L^1_\loc(\R^d)$. Then $f\in\dot W^{1,p}(\R^d)$ if and only if $m_f \in L^p_{\rm weak}(\R^{d+1}_+,\nu_p)$, and
	\begin{equation}
		\label{eq:main}
		\| \nabla f\|_{L^p(\R^d)}^p \simeq \sup_{\kappa>0} \kappa^p \nu_p(\{m_f>\kappa\}) \,.
	\end{equation}
	Moreover,
	\begin{equation}
		\label{eq:mainlim}
		\lim_{\kappa\to 0}  \kappa^p \nu_p(\{m_f>\kappa\}) = c_{d,p}\, \| \nabla f\|_{L^p(\R^d)}^p
	\end{equation}
	with
	$$
	c_{d,p} = p^{-1} \left( \frac{1}{\sqrt\pi} \, \frac{\Gamma(\frac{d+2}{2})}{\Gamma(\frac{d+3}{2})}  \right)^p.
	$$
\end{theorem}

Here and in what follows, we use the notations $\lesssim$, $\gtrsim$ and $\simeq$ to suppress constants that only depend on $d$ and $p$.

\medskip

\emph{Remarks.} (a) One motivation of this work comes from the study of trace ideal properties of commutators with singular integral operators. This concerns the case $p=d>1$ in Theorem \ref{main}. We will discuss the background in further detail in Subsection \ref{sec:traceideal}, but for now let us mention that in the context of commutator bounds, Rochberg and Semmes \cite{RoSe2} introduced a discrete analogue of the condition $m_f\in L^d_{\rm weak}(\R_+^{d+1},\nu_d)$ and raised the question of characterizing this condition more directly. Connes, Sullivan and Teleman, in the appendix of their paper \cite{CoSuTe}  together with Semmes, announced that this discrete condition is equivalent to $f\in\dot W^{1,d}(\R^d)$ and sketched a proof. The recent paper by Lord, McDonald, Sukochev and Zanin \cite{LoMDSuZa}, in conjunction with the results of Rochberg and Semmes \cite{RoSe2}, provides a complete proof under the additional assumption that $f\in L^\infty(\R^d)$. The latter proof, however, relies on rather deep results in operator theory and the theory of pseudodifferential operators. Our goal here is to provide a direct proof of \eqref{eq:main}, somewhat in the spirit of the sketch in \cite{CoSuTe}. In addition, we will prove \eqref{eq:mainlim}, which is new and which, in turn, suggests a new result in the study of trace ideal properties of commutators; see Corollary \ref{maincor}. Finally, and importantly, we generalize the above results, which are restricted to $p=d>1$ to general $p>1$. We are most grateful to Jean Van Schaftingen for suggesting this after reading an earlier version of this manuscript that only concerned the case $p=d$.\\
(b) Another motivation comes from the papers \cite{Ng,BrvSYu} by Nguyen and by Brezis, Van Schaftingen and Yung, which sparked an interest in finding characterizations of membership to Sobolev spaces that do not involve derivatives and which have led to a fast growing literature. We discuss this further in Subsection \ref{sec:sobspaces}. Here we just mention that both \eqref{eq:main} and \eqref{eq:mainlim} have their analogues in the corresponding formulas in \cite{Ng,BrvSYu}.\\
(c) The function $m_f$ appears in the characterization of other function spaces. For instance, always assuming $f\in L^1_\loc(\R^d)$, one has
\begin{align}
	\label{eq:bmo}
	f\in BMO(\R^d) && \text{iff} && \sup_{r>0} m_f(\cdot,r)\in L^\infty(\R^d) \,, && \\
	\label{eq:lp}
	f\in L^p(\R^d) + \R && \text{iff} && \sup_{r>0} m_f(\cdot,r)\in L^p(\R^d) \,, && \text{provided}\ 1<p<\infty \,, \\
	\label{eq:holder}
	f\in \dot C^s(\R^d) && \text{iff} && \sup_{r>0} r^{-s} m_f(\cdot,r)\in L^\infty(\R^d) \,, && \text{provided}\ 0<s<1 \,,\\
	\label{eq:sob}
	f\in \dot W^{1,p}(\R^d) && \text{iff} && \sup_{r>0} r^{-1} m_f(\cdot,r) \in L^p(\R^d) \,, &&
	\text{provided}\ 1<p\leq\infty \,.
\end{align}
Indeed, \eqref{eq:bmo} is simply the definition; equivalence \eqref{eq:lp} can be deduced from \cite[Proposition 8.10]{BeSh} (we are grateful to Mario Milman for showing us this argument, which improves that in \cite[Section IV.2]{St}); for \eqref{eq:holder} see \cite{Ca} and \cite[Theorem 6.3]{DVSh}, and for \eqref{eq:sob} see \cite{Cal} and \cite[Theorem 6.2]{DVSh}. Note that all these classical results involve a supremum with respect to $r>0$. Closer to the criterion in Theorem \ref{main} is the fact that
\begin{align}
	\label{eq:besov}
	f\in \dot W^{d/p,p}(\R^d) && \text{iff} && m_f \in L^p(\R^{d+1}_+,\nu_d) \,, &&
	\text{provided}\ d<p<\infty \,.
\end{align}
This is of relevance in connection with the trace ideal properties mentioned in (b) and is at least implicitly contained in \cite{RoSe2}. Seeger in \cite{Se} has identified, in great generality, function spaces defined in terms of $m_f$ as special cases of Triebel--Lizorkin spaces. As far as we can see, however, the results there are restricted to (possibly mixed) Lebesgue norms of $m_f$ and do not contain weak norms as in Theorem \ref{main}. The same applies to other derivative-free characterizations, for instance, the textbook characterization in \cite[Theorem 11.75]{Le} as well as the more recent ones in \cite{AlMaVe,YaYuZh}. It seems somewhat surprising, to us at least, that the \emph{strong} norms in \eqref{eq:besov} are replaced by a \emph{weak} norm in the endpoint case $p=d$.\\
(d) We defined $m_f$ in terms of an $L^1$-norm. Theorem \ref{main} remains valid if we use an $L^q$ norm with certain $1\leq q<\infty$, except, of course, that the implicit constant in \eqref{eq:main} may depend on $q$ and the value of $c_{d,p}$ in \eqref{eq:mainlim} changes; see Remark~\ref{mfq}.\\
(e) It is worth singling out from Theorem \ref{main} a sufficient condition for constancy of a function. Namely, if $1<p<\infty$ and if $f\in L^1_\loc(\R^d)$ satisfies $m_f\in L^p_{\rm weak}(\R^{d+1}_+,\nu_p)$ and $\liminf_{\kappa\to 0} \kappa^p\nu_p(\{m_f>\kappa\})=0$, then $f$ is constant. Related, but different conditions for constancy are discussed, for instance, in \cite{Br,BrvSYu2}.\\
(f) We have restricted ourselves in Theorem \ref{main} to first order Sobolev spaces. It is natural to expect that similar results also hold in the higher order case where in the definition of $m_f$ not only a constant, but a low degree polynomial needs to be subtracted. Many of the results mentioned in (c) extend to this case.\\
(g) It is noteworthy that the case $p=1$ is excluded in Theorem \ref{main}. Our proof shows that \eqref{eq:mainlim} remains valid for sufficiently regular functions on $\R^d$ (Lemma \ref{smoothlimit}) as well as that, if $m_f\in L^1_{\rm weak}(\R^{d+1}_+,\nu_1)$, then $f\in\dot{BV}(\R^d)$. On the other hand, it is easy to see that there is an $f\in \dot{BV}(\R^d)$ (for instance, the characteristic function of a ball) for which $m_f\not\in L^1_{\rm weak}(\R^{d+1}_+,\nu_1)$. At present, no simple characterization of the condition $m_f\in L^1_{\rm weak}(\R^{d+1}_+,\nu_1)$ seems to be available. In \cite{RoSe} Rochberg and Semmes show that for $d=1$ this space strictly contains the Besov space $\dot B^1_{1,1}(\R)$ and is strictly contained in a certain weak-type Besov space.

\medskip

The remainder of this paper is organized as follows. In the following two subsections, we present our two motivations for this study, namely trace ideal properties in Subsection \ref{sec:traceideal} and derivative-less characterizations of Sobolev spaces in Subsection \ref{sec:sobspaces}. In Section \ref{sec:lower}, we prove the inequality $\geq$ in \eqref{eq:main} and, in Section \ref{sec:upper}, we prove the inequality $\leq$ in \eqref{eq:main}, as well as \eqref{eq:mainlim}.

\medskip

It is a pleasure to dedicate this paper, in great admiration, to V.\ Maz'ya, whose work on Sobolev spaces has inspired many, including the present author.


\subsection{Trace ideal properties of commutators}\label{sec:traceideal}

Let us review the context in which the question answered by Theorem \ref{main} arises. There is a substantial literature on boundedness and compactness properties of operators 
$$
[K,f] :=Kf-fK \,,
$$
where $K$ is a Calder\'on--Zygmund singular integral operator and $f$ is a function on $\R^d$. We identify $f$ with the operator of multiplication by $f$. For simplicity, we assume that $K$ is homogeneous and translation-invariant, and that its kernel is given by a function that is smooth away from the origin with mean value zero on spheres centered at the origin. Most of the results below hold under much weaker assumptions on $K$, but the present ones do include the important special case of the Hilbert transform if $d=1$ and the Riesz transforms if $d\geq 2$. To avoid trivialities, we also assume $K\not\equiv 0$.

We will consider the operator $[K,f]$ on $L^2(\R^d)$. It is known that it is bounded if and only if $f\in BMO(\R^d)$, and it is compact if and only if $f\in CMO(\R^d)$, the closure in $BMO(\R^d)$ of compactly supported, smooth functions; see \cite{Ne,Ha} for $d=1$ and \cite{CoRoWe,Uc} for general $d$. (Our references in all of this subsection are far from complete, and in this specific case some rather concern the periodic than the whole space case.)

Having established criteria for compactness, the next questions concern quantitative versions of this property, expressed in the decay of singular values. We recall that the singular values of a compact operator $K$ in a separable Hilbert space are the square roots of the eigenvalues, counting multiplicities, of the operator $K^*K$. The Schatten spaces $\mathcal S^p$ and $\mathcal S^p_{\rm weak}$ consist of those $K$ for which the sequence of singular values belongs to $\ell^p$ and $\ell^p_{\rm weak}$, respectively. In dimension $d=1$, it was shown by Peller \cite{Pe} that, for $1\leq p<\infty$, $[K,f]\in\mathcal S^p$ if and only if $f\in\dot B^{1/p}_{p,p}(\R)$, the latter being a Besov space. (In fact, $\dot B^{1/p}_{p,p}(\R)=\dot W^{1/p,p}(\R)$ if $p>1$.) The higher dimensional case is somewhat different and it was shown by Janson and Wolff \cite{JaWo} that, for $d<p<\infty$, $[K,f]\in\mathcal S^p$ if and only $f\in\dot B^{d/p}_{p,p}(\R^d)$. The difference to the one-dimensional case is that, if $[K,f]\in\mathcal S^d$ for $d\geq 2$, then $f$ is constant. The endpoint case $d=p$ was studied in more detail by Rochberg and Semmes \cite{RoSe2} who showed that, again assuming $d\geq 2$, $[K,f]\in\mathcal S^d_{\rm weak}$ if and only if $f$ belongs to a certain space $Osc^{d,\infty}(\R^d)$. They also improved on the Lorentz scale the Janson--Wolff condition for $f$ to be constant.

One can show that a function $f$ belongs to the space $Osc^{d,\infty}(\R^d)$ if and only if $m_f\in L^d_{\rm weak}(\R^{d+1}_+,\nu_d)$. We have not found this statement in the literature. Its proof is not difficult. The space $Osc^{d,\infty}(\R^d)$ is defined by the analogue of $m_f$ with dyadic cubes instead of balls and considered not as a function of $(a,r)\in\R^{d+1}_+$, but as a sequence, indexed by dyadic cubes. By definition, $Osc^{d,\infty}(\R^d)$ is the space for which this sequence belongs to $\ell^d_{\rm weak}$.

In \cite{RoSe}, Rochberg and Semmes address the question of whether $Osc^{d,\infty}(\R^d)$ coincides with some known function space. As mentioned in Remark (h) following Theorem \ref{main}, in dimension $d=1$, they prove that this space strictly contains the Besov space $\dot B^1_{1,1}(\R)$ and is strictly contained in a certain weak-type Besov space. In dimensions $d\geq 2$, they show that $Osc^{d,\infty}(\R^d)\supset \dot W^{1,d}(\R^d)$. Modulo the equivalence of the discrete condition defining $Osc^{d,\infty}(\R^d)$ and our continuous condition, this proves $\geq$ in \eqref{eq:main}. In fact, our proof of that inequality uses some ideas from their argument, but seems to us somewhat more direct.

As mentioned in Remark (a) after Theorem \ref{main}, the remaining inclusion $Osc^{d,\infty}(\R^d)$ $\subset \dot W^{1,d}(\R^d)$ is stated as a theorem in the appendix of the paper \cite{CoSuTe} by Connes, Sullivan and Teleman, where they acknowledge a collaboration with Semmes. In particular, from our `continuous' point of view the analogue of their `discrete' equation (A3) is
\begin{align}
	\label{eq:csst}
	\limsup_{\kappa\to 0} \kappa^d \nu_d(\{ m_f>\kappa\}) \gtrsim \|\nabla f\|_{L^d(\R^d)}^d \,.
\end{align}

In view of these results, our contribution in the present paper in the case $p=d$ is, on the one hand, to fill in the details in the somewhat sketchy presentation in the appendix of \cite{CoSuTe} and, on the other hand, to show that the asymptotic bound \eqref{eq:csst} can be replaced by the limit relation \eqref{eq:mainlim}. (It is not clear to us whether one can expect a limit to exist in the discrete setting.) Our result \eqref{eq:mainlim} shows that \eqref{eq:csst} and thus \cite[(A3)]{CoSuTe} hold with $\liminf$ instead of $\limsup$. 

Motivated by these results we obtain the following condition for constancy, which strengthends those due to Janson--Wolff \cite{JaWo} and Rochberg--Semmes \cite{RoSe2}. We denote by $s_n([K,f])$ the sequence of singular values of $[K,f]$ in nonincreasing order and repeated according to multiplicities.

\begin{corollary}\label{maincor}
	Let $d\geq 2$ and $f\in CMO(\R^d)$ with
	$$
	\liminf_{N\to\infty} N^{-1+1/d} \sum_{n=1}^N s_n([K,f]) =0 \,.
	$$
	Then $f$ is constant.
\end{corollary}

Note that we do not assume a-priori that $f\in\dot W^{1,d}(\R^d)$ nor, equivalently, that $s_\cdot([K,f])\in\ell^d_{\rm weak}$. This is in contrast to the condition for constance in Remark (e) following Theorem \ref{main}. The $\liminf$ condition in the corollary is implied by the condition $\liminf_{n\to\infty} n^{1/d} s_n([K,f])=0$. Since the latter condition is satisfied whenever $s_\cdot([K,f])$ belongs to a Lorentz space $\ell^{d,q}$ with $q<\infty$, the corollary is stronger than the results in \cite{JaWo,RoSe2} (although, as we shall see, it can be proved using the methods in \cite{JaWo}).

\begin{proof}
	We assume that $f\in CMO(\R^d)$ is not constant and aim at proving that the $\liminf$ in the corollary is positive. In view of the inequality $\sum_{n=1}^N s_n([K,f]) \geq \sum_{n=1}^N |(\psi_n,[K,f]\phi_n)|$ for all orthonormal $(\psi_n),(\phi_n)\subset L^2(\R^d)$ (see, e.g, \cite[Lemma 4.1]{GoKr}), it suffices to find such orthonormal systems with $|(\psi_n,[K,f]\phi_n)| \gtrsim n^{-1/d}$. The functions $\phi_n$ are essentially constructed in \cite[Section 3]{JaWo}. Indeed, the functions there are parametrized by the points $\xi_j$ in the intersection of $\Z^d$ with a cone. Their nondecreasing rearrangement clearly behaves like $n^{-1/d}$. Moreover, it is shown there that $|\widehat{[K,f]\phi_n}| \gtrsim n^{-1/d}$ on $B_\delta(\xi_j)$. We define $\psi_n$ by $\widehat{\psi_n} := c_n \1_{B_\delta(\xi_j)} \sgn \widehat{[K,f]\phi_n}$ with $c_n$ chosen such that $\|\psi_n\|_{L^2}=1$. Then $|(\psi_n,[K,f]\phi_n) \gtrsim n^{-1/d}$, as claimed.
\end{proof}

We finally mention the recent work of Lord, McDonald, Sukochev and Zanin \cite{LoMDSuZa} which concerns a particular operator $K$, namely the sign of the Dirac operator, that plays some role in noncommutative geometry. It is shown that, if $d\geq 2$ and $f\in L^\infty(\R^d)$, then $[K,f]\in\mathcal S^d_{\rm weak}$ if and only if $f\in\dot W^{1,d}(\R^d)$. In fact, a simple approximation argument shows that the a-priori assumption $f\in L^\infty(\R^d)$ may be replaced by $f\in BMO(\R^d)$, and then we deduce by the Rochberg--Semmes result \cite{RoSe2} (in particular, the fact that the space $Osc^{d,\infty}(\R^d)$ is independent of $K$) that $Osc^{d,\infty}(\R^d) = \dot W^{1,d}(\R^d)$. This argument gives a complete proof (except for the approximation argument required to remove the boundedness assumption) of the first part of Theorem \ref{main}. It is, however, somewhat unsatisfactory that in order to deduce the real analysis statement in the theorem one needs to go through rather deep results in operator theory and the theory of pseudodifferential operators. This motivated us to look for a more direct proof, closer in spirit to the sketch in \cite{CoSuTe}. However, the intuition gained from the pseudodifferential perspective in \cite{LoMDSuZa} was also helpful in the present argument, in particular, in the proof of Lemma \ref{smoothlimit}, which contains a local version of \eqref{eq:mainlim}. We take this opportunity to thank Fedor Sukochev and Dmitriy Zanin for a fruitful discussion concerning \cite{LoMDSuZa}.


\subsection{Derivative-less characterizations of Sobolev spaces}\label{sec:sobspaces}

In \cite{BrSevSYu}, Brezis, Seeger, Van Schaftingen and Yung, unifying and extending earlier work by Nguyen \cite{Ng} (see also \cite{BoNg,BrNg,BrNg2}) and by Brezis, Van Schaftingen and Yung \cite{BrvSYu}, have shown the following fact, valid for all $1<p<\infty$ and $\gamma\in\R\setminus\{0\}$. Denoting by $\tilde\nu_\gamma$ the measure on $\mathcal X:=\{(x,y)\in\R^d\times\R^d:\ x\neq y \}$ with $d\tilde\nu_\gamma(x,y) = |x-y|^{\gamma-d}\,dx\,dy$, a function $f\in L^1_\loc(\R^d)$ belongs to $\dot W^{1,p}(\R^d)$ if and only if $(f(x)-f(y))/|x-y|^{1+\gamma/p}$ belongs to $L^p_{\rm weak}(\mathcal X,\tilde\nu_\gamma)$, and
\begin{equation}
	\label{eq:bsy}
	\|\nabla f\|_{L^p(\R^d)}^p \simeq \sup_{\kappa>0} \kappa^p \, \tilde\nu_\gamma(\{ (x,y)\in\mathcal X :\ |f(x)-f(y)|/|x-y|^{1+\gamma/p} >\kappa \}) \,.
\end{equation}
Moreover, with an explicit constant $\tilde c_{d,p}\in\R_+$,
\begin{equation}
	\label{eq:bsylim}
	\lim \kappa^p\, \tilde\nu_\gamma(\{ (x,y)\in\mathcal X :\ |f(x)-f(y)|/|x-y|^{1+\gamma/p} >\kappa \}) = |\gamma|^{-1} \tilde c_{d,p} \, \|\nabla f\|_{L^p(\R^d)}^p \,,
\end{equation}
where one considers the limit $\kappa\to\infty$ for $\gamma>0$ and $\kappa\to 0$ for $\gamma<0$. We emphasize that the paper \cite{BrSevSYu} contains many more results, including for instance a detailed analysis of the case $p=1$. Clearly, \eqref{eq:main} and \eqref{eq:mainlim} share some similarities with \eqref{eq:bsy} and \eqref{eq:bsylim}, respectively. In some vague sense one can think of $(x+y)/2$ and $|x-y|$ in \eqref{eq:bsy} as our $a$ and $r$, respectively. However, \eqref{eq:bsy} and \eqref{eq:bsylim} are completely pointwise criteria, while the function $m_f$ in \eqref{eq:main} and \eqref{eq:mainlim} involves integrals. The similarity between \eqref{eq:main}--\eqref{eq:mainlim} and \eqref{eq:bsy}--\eqref{eq:bsylim} is also reflected in our proofs, namely, most clearly, in the one of \eqref{eq:mainlim}, but also in the maximal function argument for $\geq$ in \eqref{eq:main}. We comment on this in more detail before the respective proofs.

Weak-type estimates were obtained, for instance, in \cite{GrSc}. The work \cite{BrvSYu} has led to many follow-up works and we refer to \cite{BrSevSYu} for a partial bibliography. Let us mention, in particular, \cite[Section 7]{DoMi}, where it is shown that, if $1<p<\infty$ and $f\in\dot W^{1,p}(\R^d)$, then
\begin{equation}
	\label{eq:milman}
	\|\nabla f\|_{L^p(\R^d)}^p \lesssim \liminf_{\kappa\to 0} \kappa^p \, \nu_0(\{ (a,r)\in\R^{d+1}_+ :\ |(P_r f)(a) - f(a)|/ r^{1+1/p} >\kappa \}) \,,
\end{equation}
where $P_rf = e^{-r\sqrt{-\Delta}}f$ denotes the Poisson extension of $f$. This is reminiscent of \eqref{eq:mainlim}. Note, however, that the measure in $\R^{d+1}_+$ in \eqref{eq:milman} is the usual Lebesgue measure $\nu_0$ and not $\nu_p$. Moreover, \cite{DoMi} only proves a one-sided inequality. For other related estimates with derivatives of harmonic and caloric extensions, see \cite{Doetal}.


\section{The lower bound on $\|\nabla f\|_{L^p}$}\label{sec:lower}

In this section we shall prove an upper bound on $m_f$ for $f\in\dot W^{1,p}(\R^d)$ and deduce that, if $1<p<\infty$, then $m_f\in L^p_{\rm weak}(\R^{d+1}_+,\nu_p)$. We will use the (centered) maximal function, denoted by $\mathcal M$, somewhat in the spirit of \cite{Ng} and \cite[Remark 2.3]{BrvSYu} (see also \cite[Proposition 2.1]{BrSevSYu}).

\begin{lemma}\label{maxfcn}
	If $f\in W^{1,1}_\loc(\R^d)$, then
	$$
	m_f(a,r) \lesssim r \mathcal M|\nabla f|(a) 
	\qquad\text{for all}\ (a,r)\in\R^{d+1}_+ \,.
	$$
\end{lemma}

\begin{proof}
	By the Poincar\'e inequality (see, e.g., \cite[(7.45)]{GiTr}) we have
	$$
	m_f(a,r) = \fint_{B_r(a)} \left| f(x) - \fint_{B_r(a)} f(y)\,dy \right| dx \lesssim r \fint_{B_r(a)} |\nabla f(x)|\,dx
	\ \text{for all}\ (a,r)\in\R^{d+1}_+ \,.
	$$
	Since the right side is bounded from above by $r \mathcal M|\nabla f|(a)$, the lemma follows.
\end{proof}

\begin{proof}[Proof of Theorem \ref{main}. First part]
	Let $1\!<\! p\!<\!\infty$ and $f\! \in\dot W^{1,p}(\R^d)$. Then, by Lemma~\ref{maxfcn},
	$$
	\nu_p(\{ m_f >\kappa\}) \leq \nu_p( \{(a,r)\in\R^{d+1}_+:\, r \mathcal M|\nabla f|(a)>\kappa/C\})
	\qquad\text{for all}\ \kappa>0 \,.
	$$
	For fixed $a\in\R^d$, we compute
	$$
	\int_0^\infty \1(r \mathcal M|\nabla f|(a)>\kappa/C)\, \frac{dr}{r^{p+1}} = p^{-1} \left( \frac{C\, \mathcal M|\nabla f|(a)}{\kappa} \right)^p 
	$$
	and, thus,
	$$
	 \nu_p( \{(a,r)\in\R^{d+1}_+:\, r \mathcal M|\nabla f|(a)>\kappa/C\}) \leq \frac{p^{-1} C^p}{\kappa^p} \int_{\R^d} \left( \mathcal M|\nabla f|(a) \right)^p da \,.
	$$
	The claimed bound now follows from the boundedness of the maximal function on $L^p(\R^d)$. (It is at this last step that the assumption $p>1$ enters.)
\end{proof}

We note that the above argument fits into the framework of \cite[Appendix]{DoMi}. Indeed, we deduce from Lemma \ref{maxfcn} and the boundedness of the maximal function that for the operator $T_tf(x):=t^{-1}m_f(x,t)$ the assumption \cite[(9.1)]{DoMi} is satisfied. Therefore, \cite[(9.2)]{DoMi} with $\gamma=-p$ gives the bound $\gtrsim$ in \eqref{eq:main}. We are grateful to Po-Lam Yung for this remark.

\begin{remark}\label{mfq}
	For $1\leq q<\infty$, $f\in L^q_\loc(\R^d)$ and $(a,r)\in\R^{d+1}_+$, let
	$$
	m_f^{(q)}(a,r) := \left( \fint_{B_r(a)} \left| f(x) - \fint_{B_r(a)} f(y)\,dy \right|^q dx \right)^{1/q}.
	$$
	Then clearly $m_f^{(q)}(a,r)$ is nondecreasing in $q$. We claim that, if $1<p<\infty$, if $f\in\dot W^{1,p}(\R^d)$ and if $1\leq q<dp/(d-p)$ for $p<d$ and $1\leq q<\infty$ for $p\geq d$, then $m_f^{(q)} \in L^p_{\rm weak}(\R^{d+1}_+,\nu_p)$ and
	\begin{equation}
		\label{eq:mfq}
		\sup_{\kappa>0} \kappa^p \nu_p(\{ m_f^{(q)}>\kappa \}) \lesssim \|\nabla f\|_{L^p(\R^d)}^p \,.
	\end{equation}
	(In this and the following remark, implicit constants may also depend on $q$.) Indeed, given $q$ as in the claim choose $1\leq t<\min\{p,d\}$ such that $q<dt/(d-t)$ and follow the proof of \cite[(7.45)]{GiTr}) to deduce that
	$$
	m_f^{(q)}(a,r) \lesssim r \left(  \fint_{B_r(a)} |\nabla f(x)|^t\,dx \right)^{1/t}
	\ \text{for all}\ (a,r)\in\R^{d+1}_+ \,.
	$$
	Bounding the right side by $r(\mathcal M(|\nabla f|^t)(a)^{1/t}$ we can argue as in the proof of Theorem \ref{main} above and obtain \eqref{eq:mfq}.
\end{remark}

\begin{remark}
	Another variation concerns the quantity, defined for $f\in L^q_\loc(\R^d)$ and $(a,r)\in\R^{d+1}_+$,
	$$
	\tilde m_f^{(q)}(a,r) := \left( \fint_{B_r(a)} \fint_{B_r(a)} |f(x)-f(y)|^q\,dy\,dx \right)^{1/q}.
	$$
	Then clearly $\tilde m_f^{(q)}(a,r)\geq m_f^{(q)}(a,r)$. On the other hand, by adding and subtracting the mean of $f$ on $B_r(a)$ and using the triangle inequality in $L^q$, we see that $\tilde m_f^{(q)}(a,r) \leq 2 m_f^{(q)}(a,r)$. We deduce from Remark \ref{mfq} that, if $1<p<\infty$, $f\in\dot W^{1,p}(\R^d)$ and $q$ as in that remark, then
	\begin{equation}
		\label{eq:mainalt}
		\sup_{\kappa>0} \kappa^p \nu_p(\tilde m_f^{(q)}>\kappa) \lesssim \|\nabla f\|_{L^p(\R^d)}^p \,.
	\end{equation}
	We are grateful to Jean Van Schaftingen for suggesting this argument, which simplifies significantly our original one.
\end{remark}


\section{The upper bound on $\|\nabla f\|_{L^p}$}\label{sec:upper}

In this section we prove that, if $f\in L^1_\loc(\R^d)$ satisfies $m_f\in L^p_{\rm weak}(\R^{d+1}_+,\nu_p)$ for some $1<p<\infty$, then $f\in\dot W^{1,p}(\R^d)$ and the asymptotics \eqref{eq:mainlim} hold. Our proof uses some ideas from \cite{Ng,BrvSYu}, which, in turn, is inspired by \cite{BoBrMi}.

We begin by computing the asymptotics of $\nu_p(\{ m_f>\kappa\})$ as $\kappa\to 0$. The difference from the limit relation \eqref{eq:mainlim} in Theorem \ref{main} is twofold. On the one hand, here we consider more regular functions, but on the other hand, we study a localized version of the asymptotics. The constant $c_{d,p}$ is defined in Theorem \ref{main}.

\begin{lemma}\label{smoothlimit}
	Let $f\in L^1_\loc(\R^d)$. Let $\Omega\subset\R^d$ be a convex, open set and assume that $f\in C^1(\Omega)$ and $\nabla f$ is (globally) Lipschitz on $\Omega$. Then for any bounded, open set $\omega\subset\R^d$ with $\overline\omega\subset\Omega$,
	$$
	\lim_{\kappa\to 0}  \kappa^p \nu_p(\{m_f>\kappa\}\cap(\omega\times\R_+)) = c_{d,p} \int_\omega |\nabla f|^p\,dx \,.
	$$
	If $\Omega=\R^d$ and $\nabla f$ is compactly supported, then the assertion remains valid for $\omega=\R^d$.
\end{lemma}

\begin{proof}
	\emph{Step 1.} Let us denote by $A$ the Lipschitz constant of $\nabla f$ on $\Omega$. We claim that
	\begin{equation}\label{eq:smoothlimitgoal}
		\left| m_f(a,r) - c_d' \,r\, |\nabla f(a)| \right| \leq C A\, r^2
		\qquad\text{for all}\ a\in\omega \,,\ r\leq\dist(\omega,\Omega^c)
	\end{equation}
	with
	$$
	c_d' := \frac{1}{\sqrt\pi}\ \frac{\Gamma(\tfrac{d+2}{2})}{\Gamma(\tfrac{d+3}{2})} \,. 
	$$
	To prove this, we note that
	$$
	\left| f(y)-f(x) -\nabla f(x)\cdot(y-x) \right| \leq A |x-y|^2
	\qquad\text{for all}\ x,y\in\Omega \,.
	$$
	(Here we used the convexity of $\Omega$ to write $f(y)-f(x) = \nabla f(\xi)\cdot(y-x)$ for some $\xi\in\Omega$ between $x$ and $y$.) Now let $a$ and $r$ be as in \eqref{eq:smoothlimitgoal}. Then, for all $x\in\Omega$,
	\begin{align*}
		\left| f(x) - \fint_{B_r(a)} f(y)\,dy - \nabla f(x)\cdot(x-a) \right|
		& \leq \fint_{B_r(a)} \left| f(x) - f(y) - \nabla f(x)\cdot(x-y) \right|dy \\
		& \leq A \fint_{B_r(a)} |x-y|^2\,dy \\
		& = A\left( |x-a|^2 + \tfrac{d}{d+2} r^2 \right),
	\end{align*}
	so
	\begin{align*}
		& \left| \left| f(x) - \fint_{B_r(a)} f(y)\,dy \right| - \left| \nabla f(a)\cdot(x-a) \right| \right| \\
		& \leq \left| f(x) - \fint_{B_r(a)} f(y)\,dy - \nabla f(x)\cdot(x-a) \right| + \left| \left(\nabla f(x) - \nabla f(a)\right)\cdot (x-a) \right| \\
		& \leq A\left( 2|x-a|^2 + \tfrac{d}{d+2} r^2 \right)		
	\end{align*}
	and, integrating with respect to $x\in B_r(a)$,
	\begin{align*}
		& \left| m_f(a,r) - \fint_{B_r(a)} \left| \nabla f(a)\cdot(x-a) \right|dx \right| \\
		& \leq \fint_{B_r(a)} \left| \left| f(x) - \fint_{B_r(a)} f(y)\,dy \right| - \left| \nabla f(a)\cdot(x-a) \right| \right|dx \\
		& \leq A \fint_{B_r(a)} \left( 2|x-a|^2 + \tfrac{d}{d+2} r^2 \right) dx = \tfrac{3d}{d+2}\,A\,r^2 \,. 
	\end{align*}
	It remains to note that
	$$
	\fint_{B_r(a)} \left| \nabla f(a)\cdot(x-a) \right|dx = c_d'\, r\, |\nabla f(a)| \,,
	$$
	since, for $v\in\R^d$,
	\begin{align*}
		\fint_{B_r(a)} |v\cdot(x-a)|\,dx & = \frac{\int_0^r \rho^d\,d\rho}{\int_0^r \rho^{d-1}\,d\rho}\ \frac{\int_{\Sph^{d-1}}|v\cdot\omega|\,d\omega}{|\Sph^{d-1}|} = \frac{d}{d+1}\ r\,|v|\ \frac{\int_0^\pi |\cos\theta|\sin^{d-2}\theta\,d\theta}{\int_0^\pi\sin^{d-2}\theta\,d\theta}
	\end{align*}
	and, with $B$ denoting the beta function,
	\begin{align*}
		\frac{\int_0^\pi |\cos\theta|\sin^{d-2}\theta\,d\theta}{\int_0^\pi\sin^{d-2}\theta\,d\theta}
		& = \frac{\int_{-1}^1 |t| (1-t^2)^{(d-3)/2}\,dt}{\int_{-1}^1 (1-t^2)^{(d-3)/2}\,dt}
		= \frac{\int_{0}^1 (1-u)^{(d-3)/2}\,du}{\int_{0}^1 u^{-1/2} (1-u)^{(d-3)/2}\,du} \\
		& = \frac{B(1,\frac{d-1}2)}{B(\frac12,\frac{d-1}{2})}
		= \frac{1}{\sqrt\pi }\ \frac{\Gamma(\tfrac d2)}{\Gamma(\frac{d+1}{2})} \,.
	\end{align*}
	This proves \eqref{eq:smoothlimitgoal}.
	
	\medskip
	
	\emph{Step 2.} It follows from \eqref{eq:smoothlimitgoal} that, abbreviating $\delta:=\dist(\omega,\Omega^c)$,
	\begin{align}\label{eq:smoothlimitgoalcor}
		& \{ (a,r)\in \omega\times (0,\delta]:\ c_d' r|\nabla f(a)| - CAr^2 >\kappa \}
		\subset\{ (a,r)\in\omega\times(0,\delta]:\ m_f(a,r)>\kappa \} \notag \\
		& \quad \subset \{ (a,r) \in\omega\times(0,\delta]:\ c_d' r|\nabla f(a)| + CAr^2 >\kappa \} \,.
	\end{align}
	In this step we will show that
	\begin{equation}
		\label{eq:smoothlimitgoal2}
		\lim_{\kappa\to 0} \kappa^p \nu_p(\{(a,r)\in\omega\times \R_+ :\ c_d' r|\nabla f(a)| \pm CAr^2 >\kappa \} = c_{d,p} \int_{\omega} |\nabla f(a)|^p\,da \,.
	\end{equation}
	Since
	$$
	\nu_p(\omega\times(\delta,\infty)) = |\omega| \int_\delta^\infty \frac{dr}{r^{p+1}} <\infty \,,
	$$
	this, together with \eqref{eq:smoothlimitgoalcor}, proves the assertion in the lemma.
	
	The proof of \eqref{eq:smoothlimitgoal2} is elementary, but somewhat lengthy. The basic observation is that
	\begin{align*}
		\nu_p(\{ (a,r)\in\omega\times\R_+:\ c_d' r|\nabla f(a)|>\kappa \}) & = \int_{\omega} \int_{\kappa/(c_d'|\nabla f(a)|)}^\infty \frac{dr}{r^{p+1}}\,da \\
		& = p^{-1} \int_\omega \left( \frac{c_d' |\nabla f(a)|}{\kappa} \right)^p da \\
		& = c_{d,p} \kappa^{-p} \int_\omega |\nabla f(a)|^p \,da \,.
	\end{align*}
	To include the perturbation $\pm CAr^2$, we exploit the fact that, since $r^{-p-1}$ is not integrable near $r=0$, only the small $r$ behavior of the bound in \eqref{smoothlimit} is relevant and therefore the error term $CAr^2$ is negligible compared with the main term.
	
	Let us give the details of the proof of \eqref{eq:smoothlimitgoal2}. We begin with the + case. We have $c_d' r|\nabla f(a)| + CAr^2 >\kappa$ if and only if $r>R$ with
	$$
	R:= \sqrt{\frac{\kappa}{CA} +\left(  \frac{c_d'|\nabla f(a)|}{2CA} \right)^2} -  \frac{c_d'|\nabla f(a)|}{2CA} \,.
	$$
	Thus,
	\begin{align*}
		\int_0^\infty \1( c_d' r|\nabla f(a)| + CAr^2 >\kappa )\,\frac{dr}{r^{p+1}} = p^{-1} R^{-p} \,.
	\end{align*}	
	We rewrite
	\begin{align*}
		p^{-1} R^{-p} & = p^{-1} \left( \frac{\sqrt{\frac{\kappa}{CA} +\left(  \frac{c_d'|\nabla f(a)|}{2CA} \right)^2} +  \frac{c_d'|\nabla f(a)|}{2CA}}{\frac\kappa{CA}} \right)^p \\
		& = \frac{c_{d,p}}{\kappa^p} \left( \sqrt{ \tfrac{CA}{c_d'^2} \kappa +\left(  \tfrac12 |\nabla f(a)| \right)^2} +  \tfrac12|\nabla f(a)| \right)^p.
	\end{align*}
	Thus,
	\begin{align*}
		& \nu_p( \{ (a,r) \in\omega\times\R_+ :\ c_d' r|\nabla f(a)| + CAr^2 >\kappa \} ) \\
		& = \frac{c_{d,p}}{\kappa^p} \int_{\omega} 
		\left( \sqrt{ \tfrac{CA}{c_d'^2} \kappa +\left(  \tfrac12 |\nabla f(a)| \right)^2} +  \tfrac12|\nabla f(a)| \right)^p da \,.
	\end{align*}
	Dominated convergence (recalling that $\omega$ has finite measure) implies \eqref{eq:smoothlimitgoal2} with +.
	
	We turn to the proof of \eqref{eq:smoothlimitgoal2} with -. Assuming that $\kappa< (c_d'|\nabla f(a)|)^2/(4AC)$, we have $c_d' r|\nabla f(a)| - CAr^2 >\kappa$ if and only if $R_-<r<R_+$ with
	$$
	R_\pm := \frac{c_d'|\nabla f(a)|}{2CA} \pm \sqrt{\left(  \frac{c_d'|\nabla f(a)|}{2CA} \right)^2 - \frac{\kappa}{CA}} \,.
	$$
	Thus, if $\kappa< (c_d'|\nabla f(a)|)^2/(4AC)$,
	\begin{align*}
		\int_0^\infty \1( c_d' r|\nabla f(a)| - CAr^2 >\kappa )\,\frac{dr}{r^{p+1}}
		& = p^{-1} \left( R_-^{-p} - R_+^{-p} \right).
	\end{align*}
	We rewrite, similarly as before,
	$$
	p^{-1} R_-^{-p} = \frac{c_{d,p}}{\kappa^p} \left( \sqrt{ \left(  \tfrac12 |\nabla f(a)| \right)^2 - \tfrac{CA}{c_d'^2} \kappa} +  \tfrac12|\nabla f(a)| \right)^p .
	$$
	Thus,
	\begin{align*}
		\kappa^p \nu_p(\{ (a,r)\in\omega\times\R_+:\ c_d' r|\nabla f(a)| - CAr^2>\kappa \}) = I_1(\kappa) - I_2(\kappa) \,,
	\end{align*}
	where
	\begin{align*}
		I_1(\kappa) := c_{d,p} \int_{\omega\cap\{|\nabla f|> \sqrt{4CA\kappa}/c_d'\}} \left( \left( \left(  \tfrac12 |\nabla f(a)| \right)^2 - \tfrac{CA}{c_d'^2} \kappa\right)^{1/2} +  \tfrac12|\nabla f(a)| \right)^p da  \,,\\
		I_2(\kappa) := \kappa^p p^{-1} \int_{\omega\cap\{|\nabla f|> \sqrt{4CA\kappa}/c_d'\}} \left( \tfrac{c_d'|\nabla f(a)|}{2CA} + \left( \left(  \tfrac{c_d'|\nabla f(a)|}{2CA} \right)^2 - \tfrac{\kappa}{CA} \right)^{1/2} \right)^{-p} da \,.
	\end{align*}
	By monotone convergence, $I_1(\kappa) \to c_{d,p} \|\nabla f\|_{L^p(\omega)}^p$. For $I_2(\kappa)$, we note that on $\{ |\nabla f|> \sqrt{4CA\kappa}/c_d' \}$
	$$
	\tfrac{c_d'|\nabla f(a)|}{2CA} + \left( \left(  \tfrac{c_d'|\nabla f(a)|}{2CA} \right)^2 - \tfrac{\kappa}{CA} \right)^{1/2} \geq \sqrt{\frac{\kappa}{CA}},
	$$
	so
	$$
	\kappa^p p^{-1} \left( \tfrac{c_d'|\nabla f(a)|}{2CA} + \left( \left(  \tfrac{c_d'|\nabla f(a)|}{2CA} \right)^2 - \tfrac{\kappa}{CA} \right)^{1/2} \right)^{-p} \leq \kappa^{p/2} p^{-1} (CA)^{p/2} \,.
	$$
	Thus, using again the fact that $\omega$ has finite measure, $I_2(\kappa)\to 0$. This completes the proof of \eqref{eq:smoothlimitgoal2} with -.

	\medskip
	
	\emph{Step 3.} Finally, we assume that $f$ is constant in $\{ |x|\geq R_0\}$. Applying what we have proved so far with $\omega=\{|x|<2R_0\}$, we obtain
	$$
	\lim_{\kappa\to 0} \kappa^p \nu_p(\{m_f>\kappa\}\cap \{(a,r):\ |a|<2R_0\}) = c_{d,p} \int_{|a|<2R_0} |\nabla f|^p\,dx = c_{d,p} \int_{\R^d} |\nabla f|^p\,dx \,.
	$$
	Thus, it suffices to prove that $\kappa^p \nu_p(\{m_f>\kappa\}\cap \{(a,r):\ |a|\geq 2R_0\})\to 0$. The constancy assumption on $f$ implies that $m_f(a,r)=0$ for $|a|-r\geq R_0$. Thus, it suffices to consider the intersection of $\{ m_f>\kappa\}$ with $\{R_0+r>|a|\geq 2R_0\}$.
	
	We write $f= g +c$ where $g$ is supported in $\{ |x|\leq R_0\}$ and $c$ is a constant. Then
	$$
	m_f(a,r) = \fint_{B_r(a)} \left| g(x) - \fint_{B_r(a)} g(y)\,dy \right| dx \leq 2 |B_r(a)|^{-1} \int_{\R^d} |g(x)|\,dx =: \gamma r^{-d} \,,
	$$
	and we conclude that, if $m_f(a,r)>\kappa$, then $r<(\gamma/\kappa)^{1/d}$. Thus,
	\begin{align*}
		\int_{|a|\geq 2R_0} \int_{|a|-R_0}^\infty \1(m_f(r,a)>\kappa)\,\frac{dr}{r^{p+1}}\,da
		& \leq \int_{R_0 + (\gamma/\kappa)^{1/d} >|a|\geq 2R_0} \int_{|a|-R_0}^\infty \frac{dr}{r^{p+1}}\,da \\
		& = p^{-1} \int_{R_0 + (\gamma/\kappa)^{1/d} >|a|\geq 2R_0} \frac{da}{(|a|-R_0)^p} \,. 
	\end{align*}
	When $p>d$ this is uniformly bounded in $\kappa$, when $p=d$ it is bounded by a constant times $\ln_+ (\gamma/(\kappa R_0^d))$ and when $p<d$ it is bounded by a constant times $(\gamma/\kappa)^{(d-p)/d}$. In any case, the bound, multiplied by $\kappa^p$, tends to zero as $\kappa\to 0$, as claimed.
\end{proof}

\begin{lemma}\label{approx}
	Let $1<p<\infty$. There is a constant $C_{d,p}<\infty$ such that for all $f\in L^1_\loc(\R^d)$ and all $0\leq \phi\in L^1_c(\R^d)$ with $\int \phi\,dx=1$, 
	$$
	\sup_{\kappa>0} \kappa^p \nu_p(\{m_{\phi*f}>\kappa\}) \leq C_{d,p}\ \sup_{\kappa>0} \kappa^p \nu_p(\{m_{f}>\kappa\}) \,.
	$$
\end{lemma}

\begin{proof}
	We bound, using Minkowski's inequality,
	\begin{align*}
		m_{\phi*f}(a,r) & = \fint_{B_r(a)} \left| \int_{\R^d} \phi(z) \left( f(x-z) - \fint_{B_r(a)} f(y-z)\,dy \right) dz \right|dx \\
		& \leq \int_{\R^d} \phi(z) \fint_{B_r(a)} \left| f(x-z) - \fint_{B_r(a)} f(y-z)\,dy \right|dx \,dz \\
		& = \int_{\R^d} \phi(z) m_f(a-z,r)\,dz \,.
	\end{align*}
	Passing to a norm in $L^p_{\rm weak}(\R^{d+1}_+,\nu_p)$ which is equivalent to the quasi-norm in the statement of the lemma (this is possible since $p>1$), we obtain the assertion from Minkowski's inequality.
\end{proof}

\begin{proof}[Proof of Theorem \ref{main}. Second part]
	Throughout this proof, let $1<p<\infty$ and let $f \in L^1_\loc(\R^d)$ with $m_f\in L^p_{\rm weak}(\R^{d+1}_+,\nu_p)$. We proceed in two steps.
	
	\medskip
	
	\emph{Step 1.} Let $0\leq \phi\in C^2_c(\R^d)$ with $\int \phi\,dx=1$ and set $\phi_t(x):=t^{-d} \phi(x/t)$. Note that $\phi_t*f\in C^2(\R^d)$ with $D^2(\phi_t* f) \in L^\infty_\loc(\R^d)$. Thus, by Lemma \ref{smoothlimit}, for any bounded, open set $\omega\subset\R^d$,	
	$$
	\lim_{\kappa\to 0} \kappa^p \nu_p(\{ m_{\phi_t*f}>\kappa\}\cap(\omega\times\R_+)) = c_{d,p} \int_\omega |\nabla(\phi_t * f)|^p\,dx \,.
	$$
	On the other hand, by Lemma \ref{approx},
	\begin{align*}
		\lim_{\kappa\to 0} \kappa^p \nu_p(\{ m_{\phi_t*f} \!>\kappa\}\cap(\omega\times\R_+))
		& \leq \liminf_{\kappa\to 0} \kappa^p \nu_p(\{ m_{\phi_t*f} \! >\kappa\}) \\
		& \leq \sup_{\kappa>0} \kappa^p \nu_p(\{m_{\phi_t*f} \! >\kappa\}) \\
		& \leq C_{d,p} \sup_{\kappa>0} \kappa^p \nu_p(\{m_{f}>\kappa\}) \,.
	\end{align*}
	Thus,
	$$
	\int_\omega |\nabla(\phi_t * f)|^p\,dx \leq c_{d,p}^{-1} \, C_{d,p}\ \sup_{\kappa>0} \kappa^p \nu_p(\{m_{f}>\kappa\}) =: C' \,,
	$$
	where the right side depends neither on $t$ nor on $\omega$. By monotone convergence, we conclude that $\nabla(\phi_t*f)\in L^p(\R^d)$ and
	$$
	\int_{\R^d} |\nabla(\phi_t*f)|^p\,dx \leq C' \,.
	$$
	By weak compactness (using again $p>1$), we deduce that for a sequence $t_j\to 0$, $\nabla(\phi_{t_j}*f)\rightharpoonup F$ in $L^p(\R^d)$. On the other hand, $\phi_t* f\to f$ in $L^1_\loc(\R^d)$ as $t\to 0$. Thus, for any $\Phi\in C^1_c(\R^d,\R^d)$,
	$$
	\int_{\R^d} \Phi\cdot F\,dx \leftarrow \int_{\R^d} \Phi \cdot \nabla(\phi_{t_j}*f) \,dx = - \int_{\R^d} (\nabla\cdot\Phi) \phi_{t_j}* f\,dx \to - \int_{\R^d}  (\nabla\cdot\Phi) f\,dx \,.
	$$
	This proves that $f\in\dot W^{1,p}(\R^d)$ with $\nabla f=F$. Moreover, by weak convergence,
	$$
	\int_{\R^d} |\nabla f|^p\,dx \leq \liminf_{j\to\infty}  \int_{\R^d} |\nabla(\phi_{t_j}*f)|^p\,dx \leq C' = c_{d,p}^{-1}\, C_{d,p}\ \sup_{\kappa>0} \kappa^p \nu_p(\{m_{f}>\kappa\}) \,,
	$$
	which proves the claimed upper bound on $\|\nabla f\|_{L^p(\R^d)}$ in \eqref{eq:main}.
	
	\medskip
	
	\emph{Step 2.} It remains to deduce the limit relation \eqref{eq:mainlim} for $f$. This follows by a density argument, using the fact that there is a sequence $(f_n)\subset C^2(\R^d)$ with $\nabla f_n$ compactly supported such that $\nabla f_n\to \nabla f$ in $L^p(\R^d)$; see, for instance, \cite[Theorem 11.43]{Le}. Note that $|m_{f_n} - m_f| \leq m_{f_n-f}$. Thus, for any $\delta\in(0,1)$,
	\begin{align*}
		\nu_p(\{m_f>\kappa\}) & \leq \nu_p(\{ m_{f_n} + m_{f_n-f} >\kappa\}) \\
		& \leq \nu_p(\{m_{f_n}>(1-\delta)\kappa\}) + \nu_p(\{ m_{f_n-f} \! >\delta\kappa\}), \\
		\nu_p(\{ m_{f_n}>\tfrac{\kappa}{1-\delta}\}) & \leq \nu_p(\{ m_{f} + m_{f_n-f} \! >\tfrac{\kappa}{1-\delta}\}) \\
		& \leq \nu_p(\{ m_f>\kappa\}) + \nu_p(\{ m_{f_n-f}>\tfrac{\delta\kappa}{1-\delta} \}) \,.
	\end{align*}
	By the first part of Theorem \ref{main}, we deduce that
	\begin{align*}
		\limsup_{\kappa\to0} \kappa^p \nu_p(\{m_f>\kappa\}) & \leq \limsup_{\kappa\to0} \kappa^p \nu_p(\{m_{f_n} \! > \! (1-\delta)\kappa\}) + C \delta^{-p} \|\nabla(f_n-f)\|_{L^p(\R^d)}^p , \\
		\liminf_{\kappa\to 0} \kappa^p \nu_p(\{m_{f_n}>\tfrac{\kappa}{1-\delta}\}) & \leq \liminf_{\kappa\to0} \kappa^p \nu_p(\{m_{f} \! >\kappa\}) + C \delta^{-p} (1-\delta)^p \|\nabla(f_n-f)\|_{L^p(\R^d)}^p.
	\end{align*}
	On the other hand, by Lemma \ref{smoothlimit},
	$$
	\lim_{\kappa\to0} \kappa^p \nu_p(\{m_{f_n}> \alpha \kappa\}) = \alpha^{-p} c_{d,p} \|\nabla f_n\|_{L^p(\R^d)}^p \,.
	$$
	Thus, we have shown that
	\begin{align*}
		\limsup_{\kappa\to0} \kappa^p \nu_p(\{m_f>\kappa\}) & \leq (1-\delta)^{-p} c_{d,p} \|\nabla f_n\|_{L^p(\R^d)}^p + C \delta^{-p} \|\nabla(f_n-f)\|_{L^p(\R^d)}^p \,, \\
		(1-\delta)^{p} c_{d,p} \|\nabla f_n\|_{L^p(\R^d)}^d  & \leq \liminf_{\kappa\to0} \kappa^p \nu_p(\{m_{f}>\kappa\}) + C \delta^{-p} (1-\delta)^p \|\nabla(f_n-f)\|_{L^p(\R^d)}^p \,.
	\end{align*}
	Letting first $n\to\infty$ and then $\delta\to 0$, we obtain \eqref{eq:mainlim}. This completes the proof of Theorem \ref{main}.
\end{proof}

Similarly as in Section \ref{sec:lower}, let us view the above proof from within the framework of \cite[Appendix]{DoMi}. Let again $T_tf(x):=t^{-1}m_f(x,t)$. Then \eqref{eq:smoothlimitgoal} in the proof of Lemma~\ref{smoothlimit} shows that assumption \cite[(9.3)]{DoMi} with $g=c_d'|\nabla f|$ is satisfied for $f\in C^1(\R^d)$ with $\nabla f$ compactly supported. Therefore, \cite[(9.4)]{DoMi} with $\gamma=-p$ gives \eqref{eq:mainlim} for such $f$. On the other hand, as far as we can see, \cite{DoMi} does not consider the question whether $m_f\in L^p_{\rm weak}(\R^d,\nu_p)$ implies $f \in \dot W^{1,p}(\R^d)$. We are grateful to Po-Lam Yung for this remark.


\bibliographystyle{amsalpha}

\end{document}